\documentclass[12pt]{amsart}

\usepackage{amsfonts,amsmath,amssymb,amsthm}
\usepackage[latin1]{inputenc}
\usepackage{tikz}
\usetikzlibrary{calc}
\usetikzlibrary{plotmarks}
\usepackage[all]{xy}
\usepackage{float}

\newcommand{\Cc}{\mathbb{C}} 

\newcommand{\Rr}{\mathbb{R}}

\newcommand{\Kk}{\mathbb{K}}

\renewcommand {\le}{\leqslant}
\renewcommand {\ge}{\geqslant}

\newcommand{\mymax}{\text{max}}

\newcommand{\defi}[1]{\emph{#1}}

\newcommand{\band}{\mathbb{D}}
\newcommand{\ball}{\mathbb{B}}

{\theoremstyle{plain}
\newtheorem{theorem}{Theorem}    
\newtheorem*{theorem*}{Theorem}
\newtheorem{lemma}[theorem]{Lemma}       
\newtheorem{proposition}[theorem]{Proposition}      
\newtheorem{proposition*}{Proposition} 
\newtheorem{corollary}[theorem]{Corollary}      
\newtheorem*{theoremA*}{Theorem A}
\newtheorem*{theoremAA*}{Theorem A'}
\newtheorem*{theoremB*}{Theorem B}

}
{\theoremstyle{remark}

\newtheorem*{remark*}{Remark}  
\newtheorem{remark}[theorem]{Remark}   

\newtheorem{question}{Question}
}

\title{Realization of intermediate links \\ of line arrangements}

\date{\today}

\author{Arnaud Bodin}
\email{Arnaud.Bodin@math.univ-lille1.fr}

\address{Laboratoire Paul Painlev\'e, Math\'ematiques, Universit\'e 
Lille 1, 59655 Villeneuve d'Ascq Cedex, France}



\begin{document}

\begin{abstract}
We investigate several topological and combinatorial properties of line arrangements.
We associate to a line arrangement a link $\mathcal{A} \cap S^3_r(0)$ obtained by intersecting 
the arrangement with some sphere. Several topics are discussed: (a) some link configurations can be realized 
by complex line arrangements but not by real line arrangements; (b) if we intersect the arrangements with a 
vertical band instead of a sphere, what link configurations can be obtained? (c) relations between link configurations
obtained by bands and spheres.
\end{abstract}

\maketitle


\section*{Introduction}

The topic of this paper is the study of intermediate links.
To an algebraic curve $(f(x,y)=0)$ passing through the origin we classically associate
a link $(f=0) \cap S^3_\epsilon(0)$, which is independent of $\epsilon$ for all sufficiently small
$\epsilon > 0$ (see Milnor \cite{Mi}). Another well-known and studied situation are \defi{links at infinity} when
we consider the intersection with a sphere $ S^3_R(0)$ of radius $R\gg 1$ sufficiently large.
An \defi{intermediate link} is the intersection $(f=0) \cap S^3_r(0)$, with an arbitrary $r>0$.
There is no much literature on that subject initiated by L.~Rudolph (see the surveys \cite{Ru3}, \cite{BF}
and also \cite{Ru2}, \cite{BO}, \cite{Bor}).

\medskip

We will extend and compare several concepts of intermediate links in the case of line arrangements.
\begin{itemize}
  \item We compare the configurations obtained by intersecting a complex line arrangement with a ball of $\Cc^2$
and the configurations obtained by intersecting a real line arrangement with a ball of $\Rr^2$.
  \item We compare the configurations obtained by intersecting  a real line arrangement with a ball of $\Rr^2$
and the configurations obtained by intersecting with a band of type $[-r,+r]\times \Rr$ in $\Rr^2$.
\end{itemize}

To be more precise we define the intersection graph of an arrangement in a set.
Let $\mathcal{A}$ be a real or complex line arrangement and let $B$ be a set (which will either be a ball $B^4_r(0)$ in $\Cc^2$; 
a disk $D^2_r(0)$ in $\Rr^2$; or a band $[-r,+r]\times \Rr$ in $\Rr^2$).
The \defi{intersection graph} of $\mathcal{A}$ in $B$ is the graph defined by:
\begin{itemize}
  \item one vertex associated to one line;
  \item one edge connect two vertices if the corresponding lines have their intersection inside $B$.
\end{itemize}

\begin{theorem}\ \\
\begin{minipage}{0.79\textwidth}
 This graph can be realized as the intersection graph of some complex line arrangement with a ball,
but cannot be realized as the intersection graph of a real line arrangement with a disk.
\end{minipage}
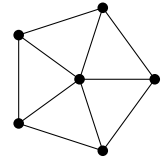
\begin{minipage}{0.20\textwidth}
\hspace*{1em}
\begin{tikzpicture}
      \coordinate (O) at (0,0);      
      \fill (O) circle (2pt);
      \foreach \i in {0,1,...,4} {
        \coordinate (A\i) at (\i*72:1); 
        \fill (A\i) circle (2pt);  
        \draw (O)--(A\i);    
      };
      \draw (A0)--(A1)--(A2)--(A3)--(A4)--cycle;
   \end{tikzpicture}   
\end{minipage}
\end{theorem}

Fix a realizable graph $G$: we can realize it by a real line arrangement $\mathcal{A}$
inside a real disk $D^2_1(0)$ of radius fixed to $1$. We define the \defi{maximal radius} $R_\mymax$
to be the maximal $r\ge 1$ such that there is no intersection point in $D^2_r(0) \setminus D^2_1(0)$.
In other words the intersection points not in $D^2_1(0)$ are as far as possible.

We also defined a maximal radius for bands (instead of disks).
The second and third parts are devoted to a numerical algorithm to compute this maximal number 
and --among other things-- prove the following results:
\begin{theorem}\ 
\begin{itemize}
  \item The maximal radius for the band problem is an algebraic number;
  \item The maximal radius for the band problem is less or equal than the maximal radius for
the disk problem.
\end{itemize}
\end{theorem}

\medskip

\emph{Acknowledgements:} I thank Patrick Popescu-Pampu for its help in the 
proof that some real configurations of lines cannot occur by using a nice extended version 
of Menelaus theorem (see part \ref{part:realvscomplex}). 
I also thank the referee for useful comments.


\bigskip
\bigskip

\part{Real and complex intermediate links of arrangements}
\label{part:realvscomplex}

Let $\Kk = \Rr$ or $\Cc$.
An \defi{arrangement} in $\Kk^n$ is a finite collection
of lines $\mathcal{A}=\{ L_i\}$.
In this part we intersect a real or complex line arrangement 
with a sphere $S^3_r(0)$ of arbitrary radius $r$.
We will detail a configuration feasible with complex lines but not feasible with real lines.

\section{The problems}
\label{sec:complexpb}

To a line arrangement $\mathcal{A}$ and a radius $r$
we associate its link $\mathcal{A} \cap S^3_r(0)$.
We can also associate an \emph{intersection graph} $G_{\mathcal{A} \cap \ball_r}$ as follows:
a vertex is associated to each line, and two vertices are connected by one edge if and only if the corresponding lines 
have intersection inside $\ball_r = \big\{ (x,y)\in \Cc^2 \mid |x|^2 + |y|^2 \le r^2 \big\}$.
In other words $G_{\mathcal{A} \cap \ball_r}$ is build as follows: 
a vertex for each knot, an edge between two linked knots.

For a given graph $G$ is such a configuration of lines exists? In other words:
\begin{question}
Can any graph $G$ (such that between any two vertices there is no or one edge)
can be realized as the link of an arrangement?  
\end{question}

One clue that it could be true is the following.
\begin{lemma}
\begin{enumerate}
  \item Any graph $G$ (such that between any two vertices there is no or one edge) can be realized 
as a quasipositive link.
  \item We may moreover suppose that each component of the link is a trivial knot and any pair of 
knots make a trivial link or a positive Hopf link. 
  \item This link is the intersection of $S^3_r(0)$ with some complex curve $f(x,y)=0$.
\end{enumerate}
\end{lemma}

\begin{proof}
We use Rudolph's theory of quasipositive link.
By induction on the number of components: when adding a trivial link
either it is unlinked with the other components so the braid word is unchanged,
either it is linked with one (or more) component: it corresponds to the addition of a word (or several)
of type: $w \sigma_i^2 w^{-1}$.
Then, by a deep result of \cite{Ru}, any quasipositive link is a transversal $\Cc$-link.
\end{proof}

\section{An example and a counter-example}
\label{sec:example}

The following example is quite interesting.
Let $G_6$ be the following graph:
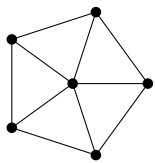
\begin{figure}[H]
   \begin{tikzpicture}
      \coordinate (O) at (0,0);      
      \fill (O) circle (2pt);
      \foreach \i in {0,1,...,4} {
        \coordinate (A\i) at (\i*72:1); 
        \fill (A\i) circle (2pt);  
        \draw (O)--(A\i);    
      };
      \draw (A0)--(A1)--(A2)--(A3)--(A4)--cycle;
   \end{tikzpicture}  
\end{figure}

\begin{proposition}
The graph $G_6$ can be realized as the intersection graph of some complex line arrangement 
but cannot be realized as the intersection graph of a real line arrangement.
\end{proposition}

\begin{proof}
\textbf{Realization as a complex line arrangement.}
Let $\omega = \exp(\frac{2\mathrm{i}\pi}{5})$. Set $P_0=(1,1)$ and $P_i = (\omega^i,\omega^{5-i})$.
Consider the $5$-lines arrangement 
$\mathcal{A}$ composed of $(P_0 P_2)$,  $(P_2 P_4)$,  $(P_4 P_1)$,  $(P_1 P_3)$,  $(P_3 P_0)$.
Finally define the sixth line $L$ of equation $(6x-4y=1)$.

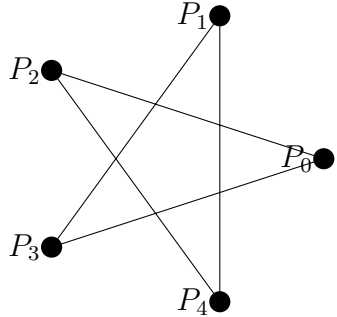
\begin{figure}[H]
   \begin{tikzpicture}[scale=2]
      \coordinate (O) at (0,0);      

      \foreach \i in {0,1,...,4} {
        \coordinate (P\i) at (\i*72:1); 
        \fill (P\i) circle (2pt);
        \node at (P\i)[left] {$P_\i$};
      };
      \draw (P0)--(P2)--(P4)--(P1)--(P3)--cycle;
        
   \end{tikzpicture}   
\end{figure}

Let the $6$-lines arrangement $\mathcal{A}'=\mathcal{A} \cup L$.
It has the following picture: all bold curves are lines, including the bold circle!
The thin circle is the sphere.

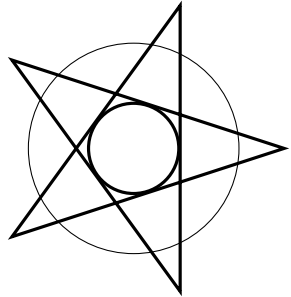
\begin{figure}[H]
   \begin{tikzpicture}[scale=2]
      \coordinate (O) at (0,0);      

      \foreach \i in {0,1,...,4} {
        \coordinate (P\i) at (\i*72:1); 
      };
      \draw[very thick] (P0)--(P2)--(P4)--(P1)--(P3)--cycle;
      \draw[very thick] circle (0.3);        
      \draw[thin] circle (0.7); 
   \end{tikzpicture} 
\end{figure}

\textbf{Fact:} $\mathcal{A}'$ has intersection graph $G_6$.
The proof is just a computation of the intersection points.

\bigskip

\textbf{Menelaus theorem for polygons.} 

One of the key-point, due to Patrick Popescu-Pampu, is a Menelaus theorem for polygons in the real plane.
The statement is given here for a pentagon, the proof for all polygons is the same as the one for triangles.
Let a line $L$ that intersects the edge lines of a pentagon $P_1,\ldots, P_5$ at points $Q_1,\ldots, Q_5$:
$Q_i$ is the intersection of $L$ with the line $(P_i P_{i+1})$.

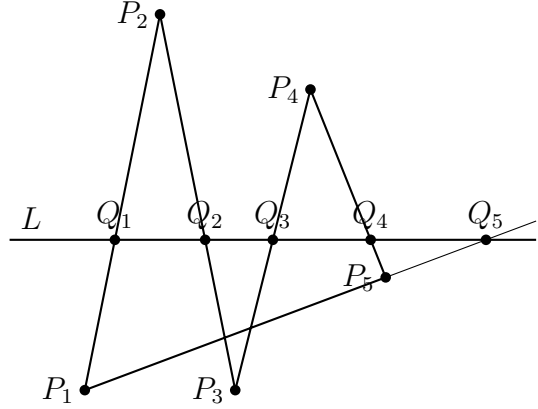
\begin{figure}[H]
   \begin{tikzpicture}[scale=1]
      \coordinate (P1) at (0,-2);      
      \coordinate (P2) at (1,3); 
      \coordinate (P3) at (2,-2);      
      \coordinate (P4) at (3,2);
      \coordinate (P5) at (4,-0.5);      

      \draw[thick] (P1)--(P2)--(P3)--(P4)--(P5)--cycle;
      \draw (P5)-- +(2,0.75);

      \coordinate (A) at (-1,0);      
      \coordinate (B) at (6,0); 
      \draw[thick] (A)--(B);

     \foreach \i in {1,2,...,5} {
        \fill (P\i) circle (2pt);
        \node at (P\i)[left] {$P_\i$};
      };

      \coordinate (Q1) at (intersection of A--B and P1--P2);
      \coordinate (Q2) at (intersection of A--B and P2--P3);
      \coordinate (Q3) at (intersection of A--B and P3--P4);
      \coordinate (Q4) at (intersection of A--B and P4--P5);
      \coordinate (Q5) at (intersection of A--B and P5--P1);

     \foreach \i in {1,2,...,5} {
        \fill (Q\i) circle (2pt);
        \node at (Q\i)[above] {$Q_\i$};
      };
        \node at (A)[above right] {$L$};
   \end{tikzpicture} 
\end{figure}

\begin{theorem}[Menelaus theorem for polygons]
$$ \frac{\overline{Q_1P_1}}{\overline{Q_1P_2}}  \times  \frac{\overline{Q_2P_2}}{\overline{Q_2P_3}} \times \cdots \times
\frac{\overline{Q_5P_5}}{\overline{Q_5P_1}} = 1.$$
\end{theorem}

The overline $\overline{AB}$ means the algebraic measure of $AB$ with respect to an orientation of the line $(AB)$.
The ratio $\frac{\overline{Q_iP_i}}{\overline{Q_iP_{i+1}}}$ is negative if and only if $Q_i$ is in the \emph{segment}
$[P_i,P_{i+1}]$ (this is independent of the chosen orientation of the line).

As a corollary we get:
\begin{corollary}
A line cannot intersect an odd number of segments of a pentagon.
\end{corollary}
Otherwise the product of the five ratios would be negative, 
that contradicts the fact that this product equals $1$.

\bigskip

\textbf{Non-realization as a real line arrangement.} 

We will apply this to our configurations. 
Suppose that $5$ lines with \emph{real equations} are disposed as follows:
there are $5$ intersection points $P_1,\ldots, P_5$ outside the ball $B_r$
and $5$ intersection points inside the ball.

%

We denote the lines as follows $\ell_i = (P_{i-1}P_{i+1})$ and $Q_i = \ell_i \cap \ell_{i+1}$
for $i$ from $0,1,\ldots,4$ (counting modulo $5$).

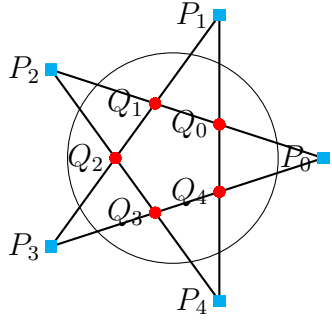
\begin{figure}[H]
   \begin{tikzpicture}[scale=2]
      \coordinate (O) at (0,0);      

      \foreach \i in {0,1,...,4} {
        \coordinate (P\i) at (\i*72:1); 
      };
      \draw[thick] (P0)--(P2)--(P4)--(P1)--(P3)--cycle;       
      \draw[thin] circle (0.7); 

      \coordinate (Q0) at (intersection of P4--P1 and P0--P2);
      \coordinate (Q1) at (intersection of P0--P2 and P1--P3);
      \coordinate (Q2) at (intersection of P1--P3 and P2--P4);
      \coordinate (Q3) at (intersection of P2--P4 and P3--P0);
      \coordinate (Q4) at (intersection of P3--P0 and P4--P1);
      \foreach \i in {0,1,...,4} {
        \draw [red] plot [only marks,mark size=1, mark=*] (Q\i);
        \node at (Q\i)[left] {$Q_\i$};
        \draw [cyan] plot [only marks,mark size=1, mark=square*] (P\i);
        \node at (P\i)[left] {$P_\i$};
      };
   \end{tikzpicture} 
\end{figure}

The first remark is that the polygon $Q_0 Q_1 \cdots Q_4$ is convex otherwise
two lines $\ell_j$, $\ell_k$ would have an intersection point, distinct from the $Q_i$,
inside the convex hull of $Q_0 Q_1 \cdots Q_4$ and hence inside the ball $B_r$.

\medskip

Each line $\ell_i$ contains two points $Q_j$ and two points $P_k$.
Two kinds of configurations for points on lines are possible.
Type (A): $P_k - Q_j - Q_{j'} - P_{k'}$ (the ball separates the two $P$)
or type (B): $P_k - P_{k'} - Q_j - Q_{j'}$ (the two $P$ are on the same side of the ball).

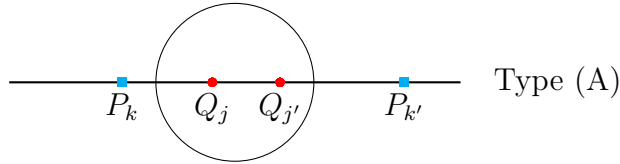
\begin{figure}[H]
   \begin{tikzpicture}[scale=1.5]
      \draw[thick] (0,0)--(4,0);  
      \draw[thin] (2,0) circle (0.7); 

      \coordinate (P) at (1,0);
      \coordinate (PP) at (3.5,0);
      \coordinate (Q) at (1.8,0);
      \coordinate (QQ) at (2.4,0);

      \draw [cyan] plot [only marks,mark size=1, mark=square*] (P);
      \draw [cyan] plot [only marks,mark size=1, mark=square*] (PP);
      \draw [red] plot [only marks,mark size=1, mark=*] (Q);
      \draw [red] plot [only marks,mark size=1, mark=*] (QQ);

      \node at (P)[below] {$P_{k}$};
      \node at (PP)[below] {$P_{k'}$};
      \node at (Q)[below] {$Q_{j}$};
      \node at (QQ)[below] {$Q_{j'}$};

      \node at (4.2,0)[right] {Type (A)};
   \end{tikzpicture} 
\end{figure}
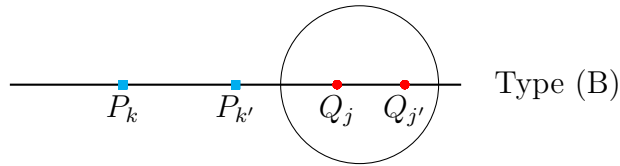
\begin{figure}[H]
   \begin{tikzpicture}[scale=1.5]
      \draw[thick] (0,0)--(4,0);  
      \draw[thin] (3.1,0) circle (0.7); 

      \coordinate (P) at (1,0);
      \coordinate (PP) at (2,0);
      \coordinate (Q) at (2.9,0);
      \coordinate (QQ) at (3.5,0);

      \draw [cyan] plot [only marks,mark size=1, mark=square*] (P);
      \draw [cyan] plot [only marks,mark size=1, mark=square*] (PP);
      \draw [red] plot [only marks,mark size=1, mark=*] (Q);
      \draw [red] plot [only marks,mark size=1, mark=*] (QQ);

      \node at (P)[below] {$P_{k}$};
      \node at (PP)[below] {$P_{k'}$};
      \node at (Q)[below] {$Q_{j}$};
      \node at (QQ)[below] {$Q_{j'}$};

      \node at (4.2,0)[right] {Type (B)};
   \end{tikzpicture} 
\end{figure}

We now prove that configurations of type (B) are associated by pairs.
Suppose for instance that we have the following configuration of type (B)
for the line $\ell_2$ where $P_1$ is farest point of the ball. 
Suppose now that the other line $\ell_0$ that contains $P_1$ is of type (A).

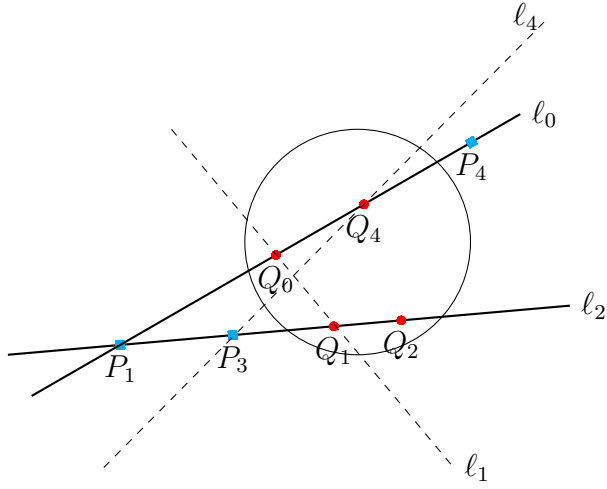
\begin{figure}[H]
   \begin{tikzpicture}[scale=1.5]

      \draw[thin] (3.1,1.0) circle (1); 
 
    \begin{scope}[rotate=5]
      \draw[thick] (0,0)--(5,0) node[right] {$\ell_2$};  

      \coordinate (P) at (1,0);
      \coordinate (PP) at (2,0);
      \coordinate (Q) at (2.9,0);
      \coordinate (QQ) at (3.5,0);

      \draw [cyan] plot [only marks,mark size=1, mark=square*] (P);
      \draw [cyan] plot [only marks,mark size=1, mark=square*] (PP);
      \draw [red] plot [only marks,mark size=1, mark=*] (Q);
      \draw [red] plot [only marks,mark size=1, mark=*] (QQ);

      \node at (P)[below] {$P_1$};
      \node at (PP)[below] {$P_3$};
      \node at (Q)[below] {$Q_1$};
      \node at (QQ)[below] {$Q_2$};      
    \end{scope}

    \begin{scope}[rotate=30,yshift=-12]
      \draw[thick] (0,0)--(5,0) node[right] {$\ell_0$};        
      \coordinate (PP) at (4.5,0);
      \coordinate (Q) at (2.5,0);
      \coordinate (QQ) at (3.4,0);

      \draw [cyan] plot [only marks,mark size=1, mark=square*] (PP);
      \draw [red] plot [only marks,mark size=1, mark=*] (Q);
      \draw [red] plot [only marks,mark size=1, mark=*] (QQ);

      \node at (PP)[below] {$P_4$};
      \node at (Q)[below] {$Q_0$};
      \node at (QQ)[below] {$Q_4$};
    \end{scope}

      \draw[dashed] (0.85,-1)--(4.8,3)  node[left] {$\ell_4$};
      \draw[dashed] (1.45,2)--(3.95,-1)  node[right] {$\ell_1$};
   \end{tikzpicture} 
\end{figure}

As $P_3 \in [P_1Q_1]$, the line $\ell_4 = (P_3Q_4)$ intersects $\ell_1 = (Q_0Q_1)$ in $[Q_0Q_1]$.
Then $P_0 = \ell_1 \cap \ell_4$ is inside the ball, that gives a contradiction.
(The same phenomenon arise if we exchange the role of $Q_1$ and $Q_2$ on the line $\ell_2$.)
As a conclusion: there is an even number of type (B) line configurations.

\medskip

Suppose now that there exists an additionnal line $L$ that intersects our five lines $\ell_i$ inside the ball
(in order to realize the graph $G_6$).
Consider the pentagon $P_0 P_1 \ldots P_4$ and its $5$ segments:
\begin{itemize}
  \item For a segment $[P_iP_{i+2}]$ supported by a line of type (A), $L$ intersects $(P_iP_{i+2})$
in the ball, hence in the segment $[P_iP_{i+2}]$.

  \item For a segment $[P_iP_{i+2}]$ supported by a line of type (B), $L$ intersects $(P_iP_{i+2})$
in the ball, hence outside the segment $[P_iP_{i+2}]$.
\end{itemize}

As there is an even number of segment supported by lines of type (B) among the $5$ segments, there is 
odd number of segments supported by a line of type (A), hence the line $L$ intersects
an odd number of segments of the pentagon. By the Menelaus theorem it provides a contradiction.
Then any line $L$ with real equation cannot intersect the $5$ lines inside the ball.
It is quite surprising that is the realm of complex number this is possible.

\end{proof}


\bigskip
\bigskip

\part{Arrangements and numerical experiments}

We study in details a variation of a problem of realisability
of intermediate links of real line arrangements. We get two problems: the realisability
and maximization of the radius. After replacing spheres
by bands, we transcript the first problem into linear inequalities.
For some examples we deduce (exact) lower bounds and (numerical) upper bounds
for the maximum radius. We end by proving that this maximum radius is an algebraic number.

\section{Statement of the problem}
\label{sec:pb}

In this part we focus on the following geometric problem, dealing with lines in the real plane.
Let $\band_r = \big\{ (x,y) \in \Rr^2 \mid |x| \le r \big\}$ be the vertical band of radius $r$.
Fix some $R\ge1$.
Given two lines, we will consider two conditions: the two lines have their intersection in $\band_1$  (the band of radius $1$);
they do not have their intersection in $\band_R$ (the band of radius $R$).

More precisely: fix $n$ and fix a graph $G$ with $n$ vertices.
The problem is to find a set of $n$ distinct lines $\{ \ell_i \}$
such that for each pair $(i,j)$ (with $i<j$): 
if an edge of $G$ connect the vertex $i$ to $j$ then $\ell_i \cap \ell_j \in \band_1$
and if no edge connect the vertex $i$ to $j$ then $\ell_i \cap \ell_j \notin \band_R$.

For a given graph $G$ and a given $R$ the first question is: is such a configuration of lines exists?
If it exists for some $R$, what is the maximal $R$ that we can choose?

\section{Linear programs}

We denote by $(y=a_ix+c_i)$ an equation of $\ell_i$.
The abscissa of the intersection $\ell_i \cap \ell_j$ is $x_{ij}= - \frac{c_i-c_j}{a_i-a_j}$.
The condition $\ell_i \cap \ell_j \in \band_1$ becomes
\begin{equation}
\label{eq:in}
 |a_i-a_j| \ge |c_i-c_j|.
\tag{$E_{i,j}^\in$}
\end{equation}
while the condition $\ell_i \cap \ell_j \notin \band_R$ becomes
\begin{equation}
\label{eq:notin}
R|a_i-a_j| < |c_i-c_j|.
\tag{$E_{i,j}^{\notin}$}
\end{equation}

These conditions can be seen as linear inequalities, after discussion 
on cases depending on the sign of $a_i-a_j$ and $c_i-c_j$ (see below, paragraph \ref{sec:impl}).

\section{Numerical results}
\label{sec:num}

We will give some examples and results for several graphs.
For the graph $G=A_4$ we conjecture numerically that $R_\mymax(A_4) = 3+2\sqrt{2}$.
More precisely: we found a configuration of lines realizable for $R=3+2\sqrt{2}$
and we numerically compute that no such configuration exists for $R=3+2\sqrt{2}+\epsilon$ 
with $\epsilon=10^{-6}$.

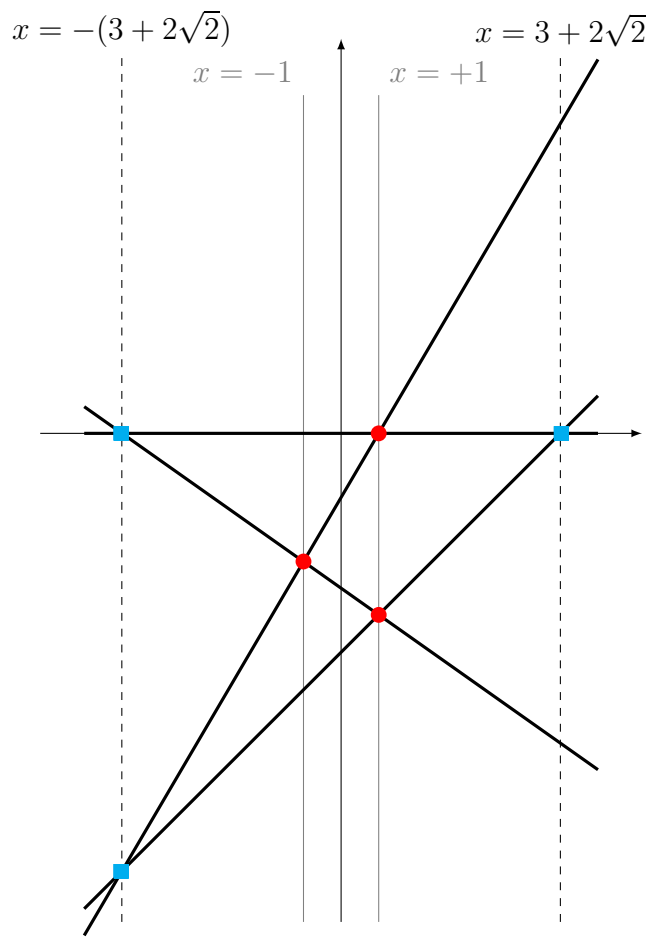
\begin{figure}[H]
   \begin{tikzpicture}
      \foreach \i in {0,1,...,3} {
         \fill (\i,0) circle (2pt);
      } 
      \foreach \i in {0,1,...,2} {
        \draw (\i,0)--({\i+1},0);
      };
      \node at (-1,0) {$A_4$};
   \end{tikzpicture}  
\end{figure}

\begin{figure}[H]
  \begin{tikzpicture}[scale=0.5,>=latex]
     \draw[->] (-8,0) -- (8,0);
     \draw[->] (0,-13) -- (0,10.5);

     \draw[gray] (1,-13)--(1,9) node[above right] {$x=+1$};
     \draw[gray] (-1,-13)--(-1,9) node[above left] {$x=-1$};

     \draw[dashed] (-5.83,-13)--(-5.83,10) node[above] {$x=-(3+2\sqrt2)$};
     \draw[dashed] (5.83,-13)--(5.83,10) node[above] {$x=3+2\sqrt2$};

     \def\rr{3+2*sqrt(2)};
     \draw[very thick] (-6.83,0)--(6.83,0);
     \draw[very thick] (-6.83,-13.36)--(6.83,9.95);
     \draw[very thick] (-6.83,0.71)--(6.83,-8.95);
     \draw[very thick] (-6.83,-12.65)--(6.83,1);

     \fill[red] (1,0) circle (6pt);
     \fill[red] (-1,-3.41) circle (6pt); 
     \fill[red] (1,-4.83) circle (6pt);

     \draw [cyan] plot [only marks,mark size=5, mark=square*] (-5.85,0);
     \draw [cyan] plot [only marks,mark size=5, mark=square*] (5.85,0);
     \draw [cyan] plot [only marks,mark size=5, mark=square*] (-5.85,-11.65);
  \end{tikzpicture}
\end{figure}

\bigskip

Here are the graph, lines and equations. 
The red dots are the intersections whose abscissa verifiy $|x| \le 1$ (here all $|x|=1$), 
the blue squares are the intersections whose abscissa have maximal $|x|\ge R$, 
for this example $|x| \ge R_\mymax=3+2\sqrt{2}$ (here all $|x|=3+2\sqrt{2}$).

\bigskip

\begin{align*}
  (\ell_1) \quad &   y=0 \\
  (\ell_2) \quad &   y=(1+\frac{\sqrt2}{2})x-1-\frac{\sqrt2}{2} \\
  (\ell_3) \quad &   y= -\frac{\sqrt2}{2}x-2-3\frac{\sqrt2}{2}\\
  (\ell_4) \quad &   y=x-3-2\sqrt2 \\
\end{align*}

\bigskip

For the graph $G=A_5$ we conjecture numerically that $R_\mymax(A_5) = 2+\sqrt3$. 

\begin{minipage}{0.4\textwidth}
  \begin{tikzpicture}
      \foreach \i in {0,1,...,4} {
         \fill (\i,0) circle (2pt);
      } 
      \foreach \i in {0,1,...,3} {
        \draw (\i,0)--({\i+1},0);
      };
      \node at (-1,0) {$A_5$};
   \end{tikzpicture}  
\end{minipage}
\begin{minipage}{0.59\textwidth}
\begin{align*}
  (\ell_1) \quad &   y=0 \\
  (\ell_2) \quad &   y=3+\sqrt5 \\
  (\ell_3) \quad &   y=x \\
  (\ell_4) \quad &   y=x+2+\sqrt5 \\
  (\ell_5) \quad &   y= -\frac{\sqrt5-1}{2}x +\frac{3+\sqrt5}{2}\\
\end{align*}  
\end{minipage}

We also find
for $G=G_1$, $R_\mymax(G_1) =3$
and for $G=G_2$, $R_\mymax(G_2) =3+2\sqrt2$.
But for both theses graphs the bound is obtained by a sequence of configuration
that tends to a ``degenerate'' configuration with two lines that are equal.

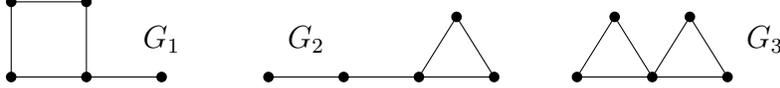
\hfil{
  \begin{tikzpicture}
      \fill (0,0) circle (2pt);
      \fill (1,0) circle (2pt);
      \fill (2,0) circle (2pt);
      \fill (0,1) circle (2pt);
      \fill (1,1) circle (2pt);

      \draw (0,0)--(0,1)--(1,1)--(1,0);
      \draw(0,0)--(2,0);

      \node at (2,0.5) {$G_1$};
   \end{tikzpicture}
\qquad 
  \begin{tikzpicture}
      \fill (0,0) circle (2pt);
      \fill (1,0) circle (2pt);
      \fill (2,0) circle (2pt);
      \fill (3,0) circle (2pt);
      \fill (2.5,0.8) circle (2pt);

      \draw (2,0)--(2.5,0.8)--(3,0);
      \draw(0,0)--(3,0);

      \node at (0.5,0.5) {$G_2$};
   \end{tikzpicture}
\qquad 
  \begin{tikzpicture}
      \fill (0,0) circle (2pt);
      \fill (1,0) circle (2pt);
      \fill (2,0) circle (2pt);

      \fill (0.5,0.8) circle (2pt);
      \fill (1.5,0.8) circle (2pt);

      \draw(0,0)--(2,0);  
      \draw (0,0)--(0.5,0.8)--(1,0)--(1.5,0.8)--(2,0);

      \node at (2.5,0.5) {$G_3$};
   \end{tikzpicture}
}

For $G=G_3$, $R_\mymax(G_3)= \alpha = 2.60\ldots$, where $\alpha$
is an algebraic number of degree $3$, that is a root of $x^3+x^2-9x-1=0$.

\begin{question}
It would be interesting to know the value of $R_\mymax(A_n)$ (where $A_n$ is the line-graph with $n$ vertices). 
Conjecturally $R_\mymax(A_n) \to 3$
as $n\to+\infty$.  
\end{question}

\section{Implementation}
\label{sec:impl}

An algorithm has been implemented in \textsc{Matlab} to decide whether for a given graph $G$ and a given $R$
a corresponding configuration of lines exists. Moreover 
--if it exists-- it gives a numeric solution.

The first step is to separate the situation in several linear problems.
To each pair $(i,j)$ with $i<j$ we have $4$ possibilities for the two signs of
$a_i-a_j$ and $c_i-c_j$. The number of pairs being $\frac{n(n-1)}{2}$.
After reduction of the case by symmetry it yields $4^{n(n-1)/2-1}$ cases.

The second step is to study each case: for a fixed condition of sign for $a_i-a_j$ and $c_i-c_j$,
the condition (\ref{eq:in}) or the condition (\ref{eq:notin})
yields a linear problem that can be solved numerically by standard tools.

This algorithm enables to find numerically $R_\mymax(G)$, by testing several $R$.
Rigorously: it first gives a value $R_0$  such that $R_0-\epsilon \le R_\mymax(G) < R_0 +\epsilon$ where $\epsilon$
is a numerical value (say $\epsilon=10^{-6}$ in practise).

Then it is possible to conjecture a value $R_1$ and the coefficients of the limit configuration and then check
that this configuration works. We then rigorously have proved $R_1 \le R_\mymax(G) < R_1 +\epsilon$.

Due to the exponential growth of the number of cases, 
we were only able to deal examples with $4$ or $5$ lines.

\begin{question}
Find an algorithm for the feasibility of band and ball problems over $\Rr$ and $\Cc$ that is efficient up to $n=10$ lines.
If a graph $G$ is feasible then compute a configuration.
\end{question}

\begin{question}
Have a rigorous proof (other than numerical) for the upper bounds of $R_\mymax$.  
\end{question}

\section{The maximum radius is an algebraic number}

Consider the coefficients of the lines $(a_1,c_1,a_2,c_2,\ldots,a_n,c_n) \in \Rr^{2n}$
as parameters. For a given graph $G$, the condition (\ref{eq:in}) and the condition (\ref{eq:notin}) for $R=1$
define a semi-algebraic set $\mathcal{S} \subset \Rr^{2n}$.
First define a function $F_1 : \mathcal{S} \to \Rr^{\frac{n(n-1)}{2}}$ by
$(a_1,c_1,\ldots) \mapsto (x_{ij})_{i<j}$ where $x_{ij}= - \frac{c_i-c_j}{a_i-a_j}$.
Secondly define $F_2 : \Rr^{\frac{n(n-1)}{2}} \to \Rr$ by
$(x_{ij}) \mapsto  \min_{|x_{ij}|>1} |x_{ij}|$ (equivalently the minimum runs over 
the pairs $(i,j)$ such that no edge of $G$ goes from $i$ to $j$).

Let $F = F_2\circ F_1 : \mathcal{S} \to \Rr$. Then by definition $R_\mymax(G) = \sup_{\mathcal{S}} F$.
By general results in semi-algebraic geometry it implies:
\begin{proposition}
Fix a graph $G$. If $R_\mymax(G)$ exists and is finite then it is an algebraic number.   
\end{proposition}


\bigskip
\bigskip

\part{Combinatorics}

In this last part we will compare two problems of realisability
and end with questions.

\section{Two real problems}

We consider two real problems. Firstly the problem already considered in section \ref{sec:pb},
dealing with the realisability of a graph as the configuration of lines within two bands.
We define a similar problem for balls, by replacing a band $\band_r$ by
the ball $\ball_r = \big\{ (x,y)\in \Rr^2 \mid x^2 + y^2 \le r^2 \big\}$.

The questions are the same. Given a graph $G$ with $n$ vertices and a real number $R\ge 1$, 
find a set of $n$ distinct lines $\{ \ell_i \}$
such that for each pair $(i,j)$ (with $i<j$): 
if an edge of $G$ connect the vertex $i$ to $j$ then $\ell_i \cap \ell_j \in \ball_1$
and if no edge connect the vertex $i$ to $j$ then $\ell_i \cap \ell_j \notin \ball_R$.

For a given graph $G$ and a given $R$ the questions are: is such a configuration of lines exists?
If it exists for some $R$, what is the maximal $R$ that we can choose?
We will compare the two problems from the combinatorial point of view.

\section{From bands to spheres}

\begin{lemma}
\label{lem:feas}
If $G$ is feasible for the bands $(\band_1,\band_R)$ (in $\Rr^2$)
then $G$ is feasible for the balls $(\ball_1,\ball_{R'})$ (in $\Rr^2$)
with $R'=R(1-\epsilon)$ (for all $\epsilon > 0$).

In particular the maximal radius $R_\mymax$ for the band problem is less or equal than the maximal radius $R_\mymax'$
for the ball problem.
\end{lemma}

\begin{remark}
In general the reciprocal is false.
For example let $G=C_5$. This graph is feasible for the ball problem $(\ball_1,\ball_{R})$
for some $R\ge 1$ but not feasible for the band problem $(\band_1,\band_{R})$
for any $R \ge 1$.

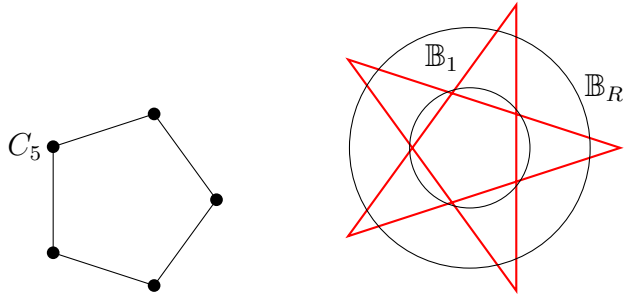
\begin{center}
   \begin{tikzpicture}[scale=1.2]
      \coordinate (O) at (0,0);      
      \foreach \i in {0,1,...,4} {
        \coordinate (A\i) at (\i*72:1); 
        \fill (A\i) circle (2pt);  
      };
      \draw (A0)--(A1)--(A2)--(A3)--(A4)--cycle;
      \node at (A2)[left] {$C_5$};
   \end{tikzpicture}\qquad\qquad
   \begin{tikzpicture}[scale=2]
      \coordinate (O) at (0,0);      

      \foreach \i in {0,1,...,4} {
        \coordinate (P\i) at (\i*72:1); 
      };
      \draw[thick,red] (P0)--(P2)--(P4)--(P1)--(P3)--cycle;
      \draw (O) circle (0.4);
      \draw (O) circle (0.8); 
      \node at (0,0.6)[left] {$\ball_1$};
      \node at (30:0.8)[right] {$\ball_R$};       
   \end{tikzpicture}   
\end{center}

Drawing a $5$-star proves the feasibility for balls.
To prove that $G$ is not feasible for bands, a first step is to remark that 
the five points of intersection in $\band_1$ draw a convex pentagon (otherwise there would be a sixth point of intersection inside
$\band_1$). The second step is to notice that for the $5$ points of intersection not in $\band_1$, at least 
$3$ of them are on the same side. So that, among this $3$ points, you can choose $Q_1$ and $Q_2$ that are on a same line of the configuration.
On this line two intersection points of the configuration are in $\band_1$ but by convexity of the pentagon they
should also be in the segment $[Q_1,Q_2]$ which is entirely out of $\band_1$. It yields a contradiction.
\end{remark}

\begin{proof}[Proof of lemma \ref{lem:feas}]
Fix $0 <\epsilon \ll 1$.
Suppose that a configuration of lines $\mathcal{L}$ realizes a graph $G$ for the band problem  $(\band_1,\band_R)$.
The transformation $(x,y) \mapsto (x,\lambda y)$ preserves equations (\ref{eq:in}) and (\ref{eq:notin}). 
So that by choosing a sufficiently small 
$0<\lambda \ll 1$ we get a ``flat'' configuration of lines $\mathcal{L}'$.
On the picture below the original configuration is on the left, the flattened one on the right.

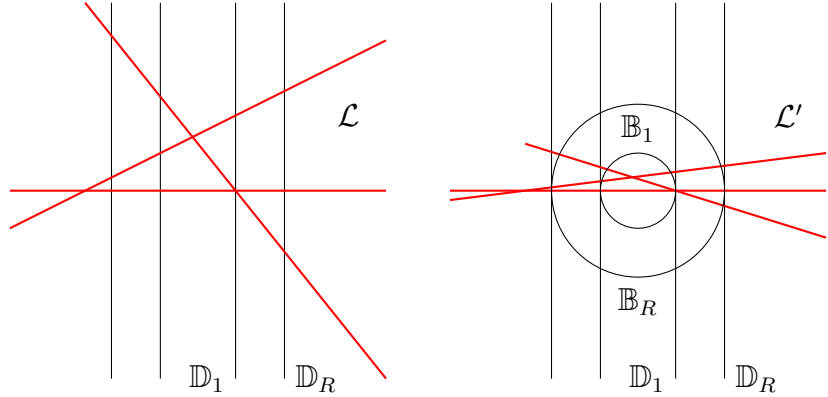
\begin{center}
   \begin{tikzpicture}[scale=0.5]
      \draw (-1,-5)--(-1,5);
      \draw (+1,-5)--(+1,5);
      \draw (-2.3,-5)--(-2.3,5);
      \draw (+2.3,-5)--(+2.3,5);
      \node at (1,-5)[left] {$\band_1$};
      \node at (2.3,-5)[right] {$\band_R$};    
      \draw[thick,red] (-5,0)--(5,0);
      \draw[thick,red] (-3,5)--(5,-5);
      \draw[thick,red] (-5,-1)--(5,4);
      \node at (4,2) {$\mathcal{L}$};
   \end{tikzpicture}\qquad
   \begin{tikzpicture}[scale=0.5]
      \draw (-1,-5)--(-1,5);
      \draw (+1,-5)--(+1,5);
      \draw (-2.3,-5)--(-2.3,5);
      \draw (+2.3,-5)--(+2.3,5);
      \node at (1,-5)[left] {$\band_1$};
      \node at (2.3,-5)[right] {$\band_R$};  
      \begin{scope}[yscale=0.25]
      \draw[thick,red] (-5,0)--(5,0);
      \draw[thick,red] (-3,5)--(5,-5);
      \draw[thick,red] (-5,-1)--(5,4);
      \end{scope}
      \node at (4,2) {$\mathcal{L}'$};  
      \draw (O) circle (1);
      \draw (O) circle (2.3); 
      \node at (0,1)[above] {$\ball_1$};
      \node at (0,-2.3)[below] {$\ball_R$};      
   \end{tikzpicture}
\end{center}

Now let  $h: (x,y) \mapsto (1-\epsilon) \cdot (x,y)$ be the homothety centred at the origin of ratio $1-\epsilon$. 
Let $P$ be a point of intersection of two lines of $\mathcal{L}'$.
Due to the flatness if $P \in \band_1$ then $h(P) \in \ball_1$ and
if $P \notin \band_R$ then $P \notin \ball_{R(1-\epsilon)}$.
So that the configuration $\mathcal{L}'$ proves the feasibility for the problem $(\ball_1,\ball_{R(1-\epsilon)})$.
\end{proof}

\section{From spheres to bands}

\begin{lemma}
Let $G$ be a graph.
If $G$ is feasible for the balls $(\ball_1,\ball_R)$ for all $R\gg1$,
then $G$ is feasible for the bands $(\band_1,\band_R)$ for all $R\gg1$.
\end{lemma}

\begin{proof}
We give a heuristic proof, and start with the case where $\complement G$ is a connected graph.
Then for a big $R$, all lines of the configurations are nearly equal (see picture below): firstly any line of the configuration should pass through 
the ball $\ball_1$. Pick a line $L_0$; any other line $L$ connected to $L_0$ in $\complement G$ should pass through $\ball_1$ that is very small compare to 
$\ball_R$, so that we think of $\ball_1$ as (nearly) a point. $L$ should also intersect $L_0$ outside $\ball_R$ at $Q$
(because $\complement G$ is a connected). So the two ``points'' of intersection $\ball_1$ and $Q$ define (nearly) the same 
line $L$ and $L_0$. Because $\complement G $ is supposed to be a connected set, it proves that all lines are (nearly) equal.

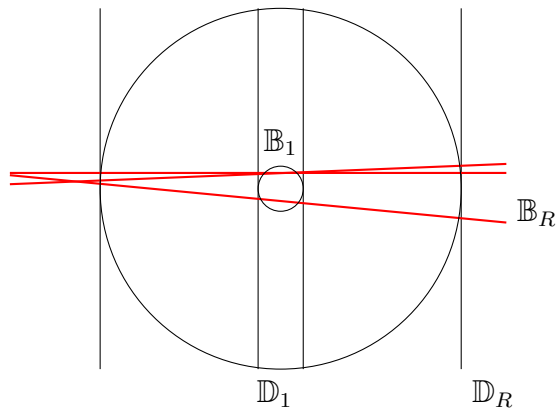
\begin{center}
   \begin{tikzpicture}[scale=0.3]
      \draw (-1,-8)--(-1,8);
      \draw (+1,-8)--(+1,8);
      \draw (-8,-8)--(-8,8);
      \draw (+8,-8)--(+8,8);
      \node at (1,-8)[below left] {$\band_1$};
      \node at (8,-8)[below right] {$\band_R$};  

      \draw[thick,red] (-12,0.7)--(10,0.7);
      \draw[thick,red] (-12,0.6)--(10,-1.5);
      \draw[thick,red] (-12,0.2)--(10,1.1);

      \draw (O) circle (1);
      \draw (O) circle (8); 
      \node at (0,1)[above] {$\ball_1$};
      \node at (10,0)[below right] {$\ball_R$};      
   \end{tikzpicture}
\end{center}

We may have supposed that $L_0$ was an horizontal line, then replacing $\ball_1$ by $\band_1$
and  $\ball_R$ by $\band_R$ proves the feasibility.

If  $\complement G$ is no longer connected, then each connected component of $\complement G$ yields a bundle of lines 
with (nearly) the same direction, any two bundles intersecting each other only in $\ball_1$. 
After choosing all directions sufficiently horizontal and replacing balls by bands, it gives the conclusion.

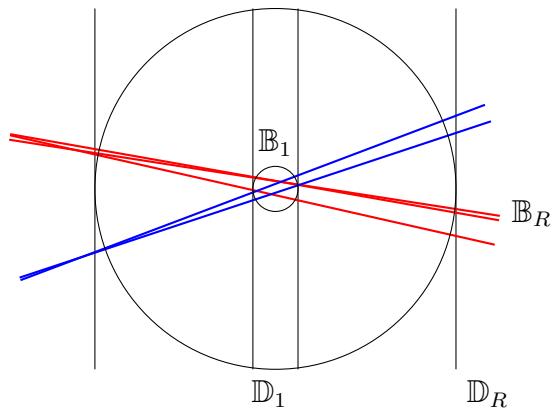
\begin{center}
   \begin{tikzpicture}[scale=0.3]
      \draw (-1,-8)--(-1,8);
      \draw (+1,-8)--(+1,8);
      \draw (-8,-8)--(-8,8);
      \draw (+8,-8)--(+8,8);
      \node at (1,-8)[below left] {$\band_1$};
      \node at (8,-8)[below right] {$\band_R$};  
      \begin{scope}[rotate=-10,yscale=0.5]
      \draw[thick,red] (-12,0.7)--(10,0.7);
      \draw[thick,red] (-12,0.6)--(10,-1.5);
      \draw[thick,red] (-12,0.2)--(10,1.1);
      \end{scope}      
      \begin{scope}[rotate=20,yscale=0.3]
      \draw[thick,blue] (-12,0.6)--(10,-1.5);
      \draw[thick,blue] (-12,0.2)--(10,1.1);
      \end{scope}

      \draw (O) circle (1);
      \draw (O) circle (8); 
      \node at (0,1)[above] {$\ball_1$};
      \node at (10,0)[below right] {$\ball_R$};      
   \end{tikzpicture}
\end{center}

\end{proof}

\begin{question}
For each class of problem (over $\Rr$ or $\Cc$) characterize feasible graphs.  
\end{question}

\begin{question}
More specifically for the complex problem with spheres, each component of a link of arrangement is in fact a true circle
(a circle in the Euclidean meaning).
For instance it is known that a Borromean ring cannot be obtained with true circles.
See \cite[Lemma 3.2]{FS} and \cite{LZ}.

 The following questions seem to be open:
\begin{itemize}
  \item Let two links of arrangements $L_1$ and $L_2$ with the same dual graph $G_1=G_2$. Does it imply $L_1$ isotopic to $L_2$?

  \item Given a graph $G$, is it possible to construct a link $L$ in $S^3$ whose components are true circles and whose dual graph is $G$?

  \item Given a link  $L$ in $S^3$ whose components are true circles, is it possible to realize $L$ as the link of a line arrangement?
\end{itemize}
\end{question}


\bigskip
\bigskip

\end{document}